\documentclass[11pt]{article}
\usepackage{amsmath,amssymb,amsthm}
\usepackage{graphicx,psfrag,rotating}
\usepackage{caption,subcaption}
\usepackage{enumerate}
\usepackage{hyperref}
\usepackage{lipsum}

\newcommand{\Z}{{{\mathbb{Z}}}}
\newcommand{\N}{{{\mathbb{N}}}}

\newcommand{\cL}{{{\mathcal{L}}}}
\newcommand{\cC}{{{\mathcal{C}}}}

\newcommand{\MCG}{{\operatorname{MCG}}}

\newtheorem{thm}{Theorem}
\newtheorem{lem}[thm]{Lemma}

\theoremstyle{definition}		
\newtheorem{defn}[thm]{Definition}

\newtheorem{alg}[thm]{Algorithm}

\theoremstyle{remark}

\makeatletter
\newcommand{\subjclass}[2]{%
  \let\@oldtitle\@title%
  \gdef\@title{\@oldtitle\footnotetext{#1 \emph{Mathematics subject classification.} #2}}%
}
\newcommand{\keywords}[1]{%
  \let\@@oldtitle\@title%
  \gdef\@title{\@@oldtitle\footnotetext{\emph{Key words and phrases.} #1.}}%
}
\makeatother

\title{Counting components of an integral lamination}
\author{S. \"Oyk\"u Yurtta\c s\footnote{Dicle University, Science
    Faculty, Mathematics Department, 21280, Diyarbak\i r, Turkey,
    e-mail: saadet.yurttas@dicle.edu.tr}\ ~and Toby
  Hall\footnote{Department of Mathematical Sciences, University of
    Liverpool, Liverpool L69 7ZL, UK,  e-mail: T.Hall@liverpool.ac.uk}\ }
\date{January 2016}

\begin{document}
\subjclass{2010}{57M50, 57N05, 20F36}
\maketitle

\begin{abstract}
We present an efficient algorithm for calculating the number of
components of an integral lamination on an $n$-punctured disk, given
its Dynnikov coordinates. The algorithm requires $O(n^2M)$ arithmetic
operations, where~$M$ is the sum of the absolute values of the
Dynnikov coordinates.
\end{abstract}

\section{Introduction}
Systems of mutually disjoint essential simple closed curves have
played a central r\^ole in the study of mapping class groups of
surfaces since the work of Dehn. Such systems are usually described
combinatorially using techniques such as train tracks or the
Dehn--Thurston coordinate system~\cite{penner}. Given such a
combinatorial description, it can be difficult to determine even
elementary properties of the system, such as the number of curves
which it contains.

In the case where the surface is an $n$-punctured disk~$D_n$, a
particularly beautiful description of such systems of curves -- or
{\em integral laminations} -- is given by the Dynnikov coordinate
system~\cite{D02}, which provides an explicit bijection from the set
of integral laminations on~$D_n$ to $\Z^{2n-4}\setminus\{0\}$.

In the case~$n=3$, the Dynnikov coordinates of an integral lamination
consist of a pair of integers, and the number of connected components
of the lamination is the greatest common divisor of these integers. No
analogous formula is known when~$n\ge 4$.

In this paper we describe an algorithm for calculating the number of
components of an integral lamination from its Dynnikov
coordinates. The algorithm proceeds by the repeated application of
three moves, each of which simplifies the lamination and either leaves
the number of components unchanged, or reduces it by a known
amount. The algorithm can be seen as complementary to that of Dynnikov
and Wiest~\cite{dynn-wiest}, which works with {\em interval identification
  systems}: combinatorial descriptions of a rather different
nature, ideally suited to their goal of comparing algebraic and
geometric notions of braid complexity. An algorithm similar to that of
Dynnikov and Wiest was given earlier by Haas and
Susskind~\cite{Haas-Susskind}, in the context of their study of the
number of curves in a system on a genus two surface described by
an integral-weighted train track.

The three moves are described, and their properties are analysed, in
Lemmas~\ref{lem:fill-puncture},~\ref{lem:erase-simple},
and~\ref{lem:untwist}, before the algorithm itself
(Algorithm~\ref{alg:algorithm}) is presented. In order to ease
implementation, the formal descriptions of the moves and the algorithm
are entirely in terms of Dynnikov coordinates rather than topological
properties of the corresponding laminations.  This method of
presentation also makes it straightforward to analyse the complexity
of the algorithm (Lemma~\ref{lem:complexity}): calculating the number
of components of an integral lamination on the $n$-punctured disk
requires $O(n^2M)$ arithmetic operations, where $M$ is the sum of the
absolute values of the Dynnikov coordinates. Here an {\em arithmetic
  operation} means adding, subtracting, comparing, taking the maximum,
or taking the minimum of two integers, each of size $O(n^2M^2)$.

The algorithm has been implemented as part of the second author's
program {\tt Dynn}, available at
\url{http://pcwww.liv.ac.uk/maths/tobyhall/software/}. In addition to
having good theoretical complexity, the algorithm is efficient in
practice. Calculating the number of components of~10000 integral
laminations on $D_{10}$ with randomly generated Dynnikov coordinates
between $-10$ and $10$, $-1000$ and $1000$, and $-100000$ and $100000$
took an average of $0.000089$, $0.00033$, and $0.00099$ seconds per
lamination on a standard notebook PC with an Intel i5 processor. On
$D_{100}$, the corresponding times were $0.0013$, $0.010$, and $0.058$
seconds per lamination.

\section{Preliminaries}
\label{sec:preliminaries}
\subsection{Integral laminations on the punctured disk}
Let $n\ge 3$, and let $D_n$ be a standard model of the $n$-punctured
disk in the plane, with the punctures arranged along the horizontal
diameter. A simple closed curve in $D_n$ is {\em inessential} if it
bounds an unpunctured disk, a once-punctured disk, or an $n$-punctured
disk, and is {\em essential} otherwise.

An {\em integral lamination} $\cL$ in $D_n$ is a non-empty union of
pairwise disjoint unoriented essential simple closed curves in $D_n$,
up to isotopy. We write $\cL_n$ for the set of integral laminations on
$D_n$.

Given an integral lamination $\cL$, we write $X(\cL)\ge 1$ for the
number of components of a representative of $\cL$. The aim
of this paper is to describe an algorithm for calculating $X(\cL)$
from the {\em Dynnikov coordinates} of $\cL$.

\subsection{The Dynnikov coordinate system}
The Dynnikov coordinate system~\cite{D02} provides, for each $n\ge 3$,
a bijection $\rho\colon \cL_n\to \Z^{2n-4} \setminus\{0\}$, which we
now define.

Construct {\em Dynnikov arcs} $\alpha_i$ ($1\le i\le 2n-4$) and
$\beta_i$ ($1\le i\le n-1$) in $D_n$ as depicted in
Figure~\ref{fig:dynn-arcs}. Given $\cL\in\cL_n$, let $L$ be a
representative of $\cL$ which intersects each of these arcs minimally
(such an~$L$ is called a {\em minimal representative} of $\cL$). Write
$\alpha_i$ (respectively $\beta_i$) for the number of intersections of
$L$ with the arc $\alpha_i$ (respectively the arc $\beta_i$). This
overload of notation will not give rise to any ambiguity, since it will
always be stated explicitly when the symbols $\alpha_i$ and $\beta_i$
refer to arcs rather than to integers.

\begin{figure}[htbp]
\begin{center}
\psfrag{1}[tl]{\begin{turn}{-90}$\scriptstyle{\alpha_1}$\end{turn}} 
\psfrag{b1}[tl]{\begin{turn}{-90}$\scriptstyle{\beta_1}$\end{turn}} 
\psfrag{b2}[tl]{\begin{turn}{-90}$\scriptstyle{\beta_{i+1}}$\end{turn}} 
\psfrag{bi}[tl]{\begin{turn}{-90}$\scriptstyle{\beta_{i}}$\end{turn}}
\psfrag{b6}[tl]{\begin{turn}{-90}$\scriptstyle{\beta_{n-1}}$\end{turn}} 
\psfrag{2i}[tl]{\begin{turn}{-90}$\scriptstyle{\alpha_{2i+2}}$\end{turn}}
\psfrag{2}[tl]{\begin{turn}{-90}$\scriptstyle{\alpha_2}$\end{turn}} 
\psfrag{2i5}[tl]{\begin{turn}{-90}$\scriptstyle{\alpha_{2i-3}}$\end{turn}}
\psfrag{2i2}[tl]{\begin{turn}{-90}$\scriptstyle{\alpha_{2i}}$\end{turn}}
\psfrag{2i3}[tl]{\begin{turn}{-90}$\scriptstyle{\alpha_{2i-1}}$\end{turn}}
\psfrag{2i1}[tl]{\begin{turn}{-90}$\scriptstyle{\alpha_{2i+1}}$\end{turn}}
\psfrag{2i4}[tl]{\begin{turn}{-90}$\scriptstyle{\alpha_{2i-2}}$\end{turn}}
\psfrag{n4}[tl]{\begin{turn}{-90}$\scriptstyle{\alpha_{2n-4}}$\end{turn}} 
\psfrag{n5}[tl]{\begin{turn}{-90}$\scriptstyle{\alpha_{2n-5}}$\end{turn}}
\psfrag{a}{$\scriptstyle{1}$} 
\psfrag{b}{$\scriptstyle{2}$} 
\psfrag{c}{$\scriptstyle{i}$} 
\psfrag{d}{$\scriptstyle{i+1}$} 
\psfrag{e}[r]{$\scriptstyle{i+2}$} 
\psfrag{f}{$\scriptstyle{n-1}$} 
\psfrag{g}{$\scriptstyle{n}$} 
\includegraphics[width=0.9\textwidth]{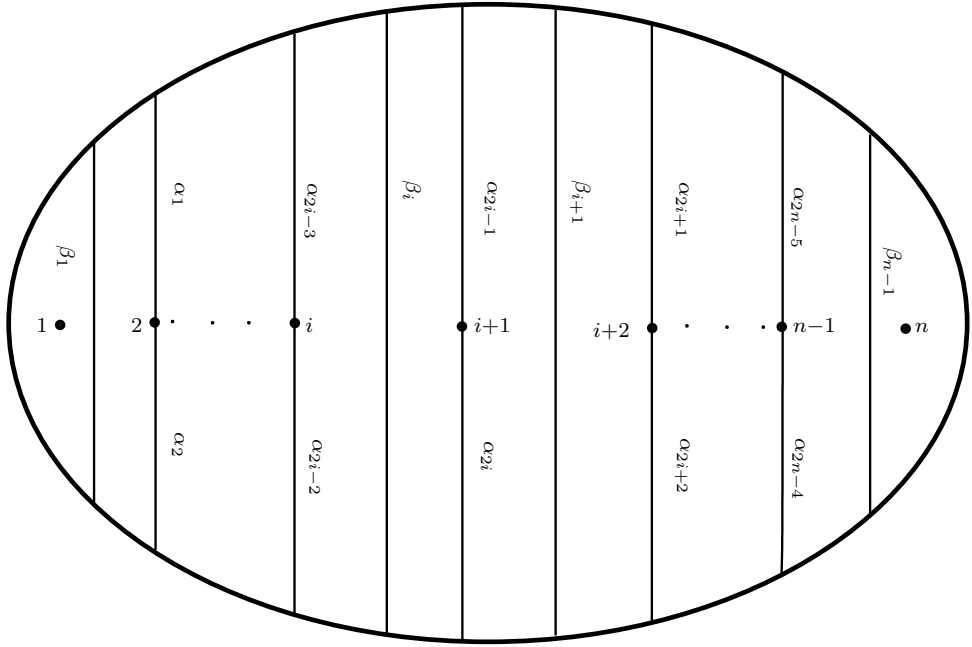}
\caption{The arcs $\alpha_i$ and $\beta_i$}
\label{fig:dynn-arcs}
\end{center}
\end{figure}


The {\em Dynnikov coordinate function}
$\rho\colon\cL_n\to\Z^{2n-4}\setminus\{0\}$ is defined by 
\[
\rho(\cL) = (a;\,b) = (a_1,\ldots,a_{n-2};\,b_1,\ldots,b_{n-2}),
\]
where 
\begin{equation}
\label{eq:dynn-coords}
a_i = \frac{\alpha_{2i}-\alpha_{2i-1}}{2} \qquad\text{and}\qquad
b_i = \frac{\beta_i - \beta_{i+1}}{2} \qquad
\end{equation}
for $1\le i\le n-2$.

The intersection numbers $\alpha_i$ and $\beta_i$ (and hence the
integral lamination~$\cL$) can be recovered from the
Dynnikov coordinates $(a;\,b)$ using the following formulae (see for
example~\cite{paper1}): 
\begin{align}
\label{eq:beta}
\beta_i &= 2\max_{1\le k\le n-2} \left( |a_k| + b_k^+ +
\sum_{j=1}^{k-1}b_j \right) - 2\sum_{j=1}^{i-1} b_j, 
\\
\label{eq:alpha}
\alpha_i &=
\left\{
\begin{array}{ll}
(-1)^i a_{\lceil i/2\rceil} + \frac{\beta_{\lceil i/2\rceil}}{2} &
  \quad\text{ if }b_{\lceil i/2\rceil} \ge 0, \\
(-1)^i a_{\lceil i/2\rceil} + \frac{\beta_{1+\lceil i/2\rceil}}{2} &
  \quad\text{ if }b_{\lceil i/2\rceil} \le 0. \\
\end{array}
\right.
\end{align}
Here $x^+$ denotes $\max(x,0)$, and $\lceil x\rceil$ denotes the
smallest integer which is not less than~$x$.

\medskip\medskip

We next mention some relevant facts about the interpretation of the
Dynnikov coordinates. Let $1\le i\le n-2$, and let $\Delta_i$ denote
the subset of~$D_n$ bounded by the arcs $\beta_i$ and
$\beta_{i+1}$. Let $L$ be a minimal representative of~$\cL$, and
consider the connected components of $L\cap\Delta_i$. By minimality,
each such component is of one of four types:
\begin{itemize}
\item A {\em right loop} component, which has both endpoints on the
  arc $\beta_i$ and intersects both of the arcs $\alpha_{2i-1}$ and
  $\alpha_{2i}$;
\item A {\em left loop} component, which has both endpoints on the arc
  $\beta_{i+1}$ and intersects both of the arcs $\alpha_{2i-1}$ and
  $\alpha_{2i}$;
\item An {\em above} component, which has one endpoint on each of the
  arcs $\beta_i$ and $\beta_{i+1}$, and intersects the arc
  $\alpha_{2i-1}$ but not the arc $\alpha_{2i}$; or
\item A {\em below} component, which has one endpoint on each of the
  arcs $\beta_i$ and $\beta_{i+1}$, and intersects the arc
  $\alpha_{2i}$ but not the arc $\alpha_{2i-1}$.
\end{itemize}
Clearly there cannot be both left loop and right loop components.
It follows immediately from~(\ref{eq:dynn-coords}) that there
are~$|b_i|$ loop components, which are left loops if $b_i<0$, and
right loops if $b_i>0$; and that $2a_i$ is the difference between the
number of below components and the number of above components.

Now suppose that $2\le i\le n-2$, and that $b_{i-1}<0$ and $b_i>0$, so
that $\Delta_{i-1}$ contains left loop components and $\Delta_i$
contains right loop components. Lemma~\ref{lem:elementary} below says
that these loop components join up to give simple closed curves
provided that $a_{i-1}=a_1$. 

\begin{defn}[Elementary curve]
An essential simple closed curve in $D_n$ is {\em elementary} (about
punctures $i$ and $i+1$) if it is isotopic to a simple closed curve
which is contained in $\Delta_{i-1}\cup\Delta_i$ for some $1\le i\le
n-1$. (Here $\Delta_0$ is the subset of
$D_n$ containing puncture~1  and bounded by
$\beta_1$; and $\Delta_{n-1}$ is the subset of $D_n$ containing
puncture~$n$ and bounded by $\beta_{n-1}$.)
\end{defn}

\begin{lem}
\label{lem:elementary}
Let~$L$ be a minimal representative of an integral lamination
$\cL\in\cL_n$ with Dynnikov coordinates $(a;\,b)$.  Let $2\le i\le
n-2$, and suppose that $b_{i-1}<0$, $b_i>0$, and $a_{i-1}=a_i$.

Then~$L$ has $\min(-b_{i-1},\,b_i)$ components which are elementary
curves about punctures $i$ and $i+1$.
\end{lem}
\begin{proof}
We suppose without loss of generality that $b_i \ge -b_{i-1}$ (otherwise
reflect in the arc $\beta_i$) and that $a_{i-1}=a_i\ge 0$
(otherwise reflect in the horizontal diameter of the disk).

Let there be~$A\ge 0$ above components of $L\cap\Delta_i$. Then
the endpoints of components of $L\cap\Delta_i$ on $\beta_i$,
ordered from top to bottom, come: $A$ from above components;
$2b_i$ from $b_i$ right loop components; and
$A+2a_i$ from below components.

Since these endpoints must agree with those of components of
$L\cap\Delta_{i-1}$ on $\beta_i$, there are also $2(A+a_i+b_i) =
2(A+a_{i-1}+b_i)$ of these, which, again from top to bottom, come:
$A+b_i+b_{i-1}$ from above components; $-2b_{i-1}$ from $-b_{i-1}$
left loop components; and $A+b_i+b_{i-1}+2a_i$ from below components.

Therefore the $j^\text{th}$ outermost right loop component in
$\Delta_i$ joins the endpoint in position $A+j$ to that in
position $A+2b_i+1-j$ ($1\le j\le b_i$); and the $k^\text{th}$
outermost left loop component in $\Delta_{i-1}$ joins the endpoint in
position $A+b_i+b_{i-1}+k$ to that in position $A+b_i-b_{i-1}+1-k$
($1\le k\le -b_{i-1}$). It follows that the $k^\text{th}$ outermost left loop
component shares its endpoints with the $(b_i+b_{i-1}+k)^\text{th}$
outermost right loop component, yielding $-b_{i-1}$ elementary curve
components as required.
\end{proof}

\subsection{The action of the braid group}
The mapping class group $\MCG(D_n)$ of $D_n$ is isomorphic to the
$n$-braid group~$B_n$ modulo its center~\cite{emil2}, so that
elements of $\MCG(D_n)$ can be represented in terms of the Artin braid
generators $\sigma_i$ ($1\le i\le n-1$). In this paper we adopt the
convention of Birman's book~\cite{birman}, that $\sigma_i$ exchanges
punctures~$i$ and~$i+1$ in the {\em counterclockwise} direction.

The action of $\MCG(D_n)$ on $\cL_n$ can be calculated using the {\em
  update rules} of Theorem~\ref{thm:update} below (see for
example~\cite{D02, M06, or08, paper1, paper2}), which describe how
Dynnikov coordinates transform under the action of the Artin
generators and their inverses. In this theorem statement we again use
the notation $x^+$ to denote $\max(x,0)$.

\begin{thm}[Update rules]
\label{thm:update}
Let $\cL\in\cL_n$ have Dynnikov coordinates $(a;\,b)$, and let $1\le
i\le n-1$.
\begin{enumerate}[(a)]
\item Let $(a';\,b')$ be the Dynnikov coordinates of the integral
  lamination $\sigma_i(\cL)$. Then $a_j' = a_j$ and $b_j' = b_j$ for
  all $j\not\in\{i-1,i\}$, and

\noindent $\bullet$ \ If $i=1$ then
\begin{equation}
\label{eq:pos-update-left}
\begin{aligned}
a_1' &= b_1 - \max(0,\,\,b_1^+-a_1),\\
b_1' &= b_1^+ - a_1.
\end{aligned}
\end{equation}

\noindent $\bullet$ \ If $2\le i\le n-2$ then
\begin{equation}
\label{eq:pos-update}
\begin{aligned}
a_{i-1}' &= \max(a_{i-1} + b_{i-1}^+,\,\, a_i + b_{i-1}),\\
a_i' &= b_i - \max(-a_{i-1},\,\,b_i^+ - a_i),\\
b_{i-1}' &= a_i + b_{i-1} + b_i - \max(a_{i-1} + b_{i-1}^+ + b_i^+,\,\,
a_i + b_{i-1}),\\
b_i' &= \max(a_{i-1} + b_{i-1}^+ + b_i^+,\,\, a_i + b_{i-1}) - a_i.
\end{aligned}
\end{equation}

\noindent $\bullet$ \ If $i=n-1$ then
\begin{equation}
\label{eq:pos-update-right}
\begin{aligned}
a_{n-2}' &= \max(a_{n-2} + b_{n-2}^+,\,\, b_{n-2}),\\
b_{n-2}' &= b_{n-2} - a_{n-2} - b_{n-2}^+.
\end{aligned}
\end{equation}

\item Let $(a'';\,b'')$ be the Dynnikov coordinates of the integral
  lamination $\sigma_i^{-1}(\cL)$. Then $a_j'' = a_j$ and $b_j'' = b_j$ for
  all $j\not\in\{i-1,i\}$, and

\noindent $\bullet$ \ If $i=1$ then
\begin{equation}
\label{eq:neg-update-left}
\begin{aligned}
a_1'' &= \max(0,\,\,a_1+b_1^+) - b_1,\\
b_1'' &= a_1 + b_1^+.
\end{aligned}
\end{equation}

\noindent $\bullet$ \ If $2\le i\le n-2$ then
\begin{equation}
\label{eq:neg-update}
\begin{aligned}
a_{i-1}'' &= \min(a_{i-1}-b_{i-1}^+,\,\,a_i-b_{i-1}),\\
a_i'' &= \max(a_{i-1},\,\, a_i + b_i^+) - b_i,\\
b_{i-1}'' &= -a_i + b_{i-1} + b_i - \max(-a_{i-1} + b_{i-1}^+ +
b_i^+,\,\, -a_i + b_{i-1}),\\
b_i'' &= a_i + \max(b_{i-1}^+ + b_i^+ - a_{i-1},\,\, b_{i-1}-a_i).
\end{aligned}
\end{equation}

\noindent $\bullet$ \ If $i=n-1$ then
\begin{equation}
\label{eq:neg-update-right}
\begin{aligned}
a_{n-2}'' &= \min(a_{n-2} - b_{n-2}^+  ,\,\, -b_{n-2}  ),\\
b_{n-2}'' &= a_{n-2} + b_{n-2} - b_{n-2}^+.
\end{aligned}
\end{equation}
\end{enumerate}
\end{thm}

\section{The algorithm}
\label{sec:algorithm}
Given $(a;\,b)\in\Z^{2n-4}\setminus\{0\}$, we write $X(a;\,b):=X(\cL)$, where
$\cL = \rho^{-1}(a;\,b)\in\cL_n$. Algorithm~\ref{alg:algorithm} below
computes $X(a;\,b)$ from $(a;\,b)$.

\subsection{The case $n=3$}
The following result is an expression of the well-known fact that the
braid group $B_3$ acts on $\cL_3$ by Euclid's algorithm.
\begin{lem}
\label{lem:3-punctures}
Let $(a_1;\,b_1)\in\Z^2\setminus\{0\}$. Then $X(a_1;\,b_1) = \gcd(a_1,\,b_1)$.
\end{lem}
\begin{proof}
It can easily be seen from Theorem~\ref{thm:update} that the action of
the Artin generators of~$B_3$ on $\cL_3$, when expressed in Dynnikov
coordinates, preserves $\gcd(a_1,\,b_1)$. 

Let~$L$ be a minimal representative of~$\cL$, and let~$C$ be any
component of~$L$. Since~$C$ is essential, it bounds a disk
containing~2 of the~3 punctures, and hence there is a braid $\sigma\in
B_3$ such that $\sigma(C)$ is an elementary curve about punctures $1$
and $2$. Since the components of $\sigma(L)$ are disjoint, it consists
of $X(a_1;\,b_1)$ elementary curves about these punctures, and hence
$\rho(\sigma(L)) = (0;\, X(a_1;\,b_1))$. Therefore $X(a_1;\,b_1) =
\gcd(0,\,X(a_1;\,b_1)) = \gcd(a_1,\,b_1)$ as required.
\end{proof}

\subsection{Extended Dynnikov coordinates}
The first step of Algorithm~\ref{alg:algorithm} is to add two
``dummy'' punctures, one to the left and one to the right of the
existing punctures.

The motivation for this is that one of the moves of the algorithm
fills in a puncture (see Section~\ref{sec:filling}). If this were done
without the dummy punctures, it could result in boundary-parallel
components, and therefore take us out of the realm of integral
laminations. While it would be possible to calculate, and compensate
for, the number of such boundary-parallel components, this would
involve an inversion of Dynnikov coordinates using~(\ref{eq:beta})
and~(\ref{eq:alpha}) each time that a puncture is filled in, and would
therefore decrease the efficiency of the algorithm.

An additional benefit of the dummy punctures is to simplify the
statement of the algorithm. Once the dummy punctures have been added,
the braid group~$B_n$ acts on the central~$n$ punctures of an
$(n+2)$-punctured disk, so that the update rules are always given
by~(\ref{eq:pos-update}) and~(\ref{eq:neg-update}), avoiding the need
for separate end
cases~\eqref{eq:pos-update-left},~\eqref{eq:pos-update-right},~\eqref{eq:neg-update-left},
and~\eqref{eq:neg-update-right}.

Introducing the dummy punctures involves the extension of
Figure~\ref{fig:dynn-arcs} to include additional punctures labelled
$0$ and $n+1$, and additional arcs $\beta_0$ (between punctures $0$
and $1$); $\alpha_{-1}$ and $\alpha_0$ (each with an endpoint on
puncture~$1$); $\beta_n$ (between punctures $n$ and $n+1$); and
$\alpha_{2n-3}$ and $\alpha_{2n-2}$ (each with an endpoint on
puncture~$n$). Additional coordinates $a_0$, $b_0$, $a_{n-1}$ and
$b_{n-1}$ can then be defined using~(\ref{eq:dynn-coords}).

To describe an integral lamination $\cL\in\cL_n$ in these
extended coordinates, observe that we have $\alpha_{-1}=\alpha_0$,
$\alpha_{2n-3}=\alpha_{2n-2}$, and $\beta_0=\beta_n=0$, so that
\begin{equation}
\label{eq:extend}
\begin{aligned}
a_0 &= a_{n-1} = 0,\\
b_0 &= - \max_{1\le k\le n-2}
\left(
|a_k| + b_k^+ + \sum_{j=1}^{k-1} b_j
\right), \quad\text{and}\\
b_{n-1} &= -b_0 - \sum_{j=1}^{n-2} b_j
,
\end{aligned}
\end{equation}
using $b_0 = (\beta_0-\beta_1)/2 = -\beta_1/2$, $b_{n-1} =
(\beta_{n-1}-\beta_n)/2 = \beta_{n-1}/2$, and~(\ref{eq:beta}).

As mentioned above, we will always consider the action of the braid
group $B_n$ on integral laminations on this $(n+2)$-punctured disk,
using~(\ref{eq:pos-update}) and~(\ref{eq:neg-update}) for $1\le i\le n-1$. 

\begin{defn}[Central lamination]
An integral lamination~$\cL$ on the $(n+2)$-punctured disk is said to
be {\em central} if it satisfies $\beta_0=\beta_n=0$.
\end{defn}

 We will see that the algorithm moves all preserve the property of
 centrality, so that we will have
\[
a_0 = a_{n-1} = 0, \quad b_0\le 0,\quad \text{and}\quad b_{n-1}\ge 0
\]
throughout.

\subsection{The complexity function}
Given Dynnikov coordinates $(a;\,b) =
(a_0,\ldots,a_{n-1};\,b_0,\ldots,b_{n-1}) \in \Z^{2n}\setminus\{0\}$
of a central lamination,
we write $n(a;\,b) = n$, and define $i(a;\,b)$ by
\begin{enumerate}[a)]
\item $i(a;\,b) = 0$ if $b_i = 0$ for any~$i$, and otherwise
\item $i(a;\,b)$ is the smallest~$i\in\{1,\ldots,n-1\}$ with $b_i>0$.
\end{enumerate}
Note that such a smallest~$i$ must exist in case~b), since
$b_{n-1}>0$; and that $b_{i-1}<0$ since $b_0<0$.

\medskip

Progress through the algorithm is measured by decrease, in the
lexicographic order, of the {\em complexity}
\[
\cC(a;\,b) = \left(n(a;\,b),\,\, \sum_{i=0}^{n(a;\,b)-1}|b_i|,\,\,\,i(a;\,b)\right)\in\N^3.
\]

\subsection{The moves}
In this section we describe and analyse each of the three moves of the
algorithm: Filling in a puncture; Erasing elementary components; and
Untwisting. While the interpretation of each of these moves is
explained briefly at the start, and is clarified in the proofs of the
relevant lemmas, the formal descriptions of the moves and the
statements of their properties are given entirely in terms of Dynnikov
coordinates, with the intention of making it easier for a reader to
implement them.

Examples of the application of each of these moves can be found in the
extended example of Section~\ref{sec:example}.
\subsubsection{Filling in a puncture}
\label{sec:filling}
This move is applied when some $b_i=0$, so that a minimal
representative of the lamination has no loops about
puncture~$i+1$. The minimal representative therefore remains minimal
when the puncture is filled in.

\begin{lem}
\label{lem:fill-puncture}
Let $(a;\,b)\in\Z^{2n}$ be the Dynnikov coordinates of a central
integral lamination~$\cL$, with $n>3$. 

Suppose that $b_i=0$ for some~$i$. Let $(a';\,b')\in\Z^{2n-2}$ be
obtained from $(a;\,b)$ by erasing the coordinates $a_i$ and $b_i$ (so
that $a'_j = a_j$ and $b'_j = b_j$ for $j<i$, while $a'_j = a_{j+1}$
and $b'_j=b_{j+1}$ for $j\ge i$).

Then $(a';\,b')$ are the Dynnikov coordinates of a central integral
lamination with $X(a';\,b') = X(a;\,b)$ and $\cC(a';\,b') < \cC(a;\,b)$. 
\end{lem}
\begin{proof}
Let~$L\subset D_{n+2}$ be a minimal representative of~$\cL$. We
write~$L'$ for the same disjoint union of simple closed curves,
regarded as a subset of the $(n+1)$-punctured disk $D_{n+1}$ obtained by
filling in puncture~$i+1$. Take Dynnikov arcs in $D_{n+1}$ given by
$\alpha_j' = \alpha_j$ for $-1\le j\le 2i-2$; $\alpha_j' =
\alpha_{j+2}$ for $2i-1\le j\le 2n-4$; $\beta_j'=\beta_j$ for $0\le
j\le i$; and $\beta_j'=\beta_{j+1}$ for $i+1\le j\le n-1$.  We shall
show that
\begin{enumerate}[a)]
\item Each component of~$L'$ is essential in $D_{n+1}$, so that $L'$ is a
  representative of an integral lamination~$\cL'$;
\item $L'$ is a minimal representative of $\cL'$; and
\item $\cL'$ is central.
\end{enumerate}

It follows from~b) that the Dynnikov coordinates $(a';\,b')$ of~$\cL'$
are as given in the statement of the lemma. Since $X(a;\,b)$ and
$X(a';\,b')$ are both equal to the number of components of~$L$, and
$n(a';\,b')=n(a;\,b)-1$, this will establish the result.

\begin{enumerate}[a)]
\item Since~$\cL$ is central, $L'$ does not intersect the arcs
  $\beta'_0=\beta_0$ or $\beta'_{n-1}=\beta_n$, and hence has no
  components which bound an $(n+1)$-punctured disk. Since every
  component of~$L$ bounds a disk containing at least two punctures, no
  component of~$L'$ can bound an unpunctured disk; and if there were a
  component bounding a once-punctured disk, it would coincide with a
  component of~$L$ bounding a disk containing puncture~$i+1$ and one
  other puncture. If this puncture were to the right (respectively
  left) of puncture~$i+1$, then $L$ would have a left (respectively
  right) loop component in $\Delta_i$, and we would have $b_i<0$
  (respectively $b_i>0$). Since $b_i=0$, this is impossible.

\item Suppose for a contradiction that $L'$ is not a minimal
  representative of $\cL'$. Then there is a Dynnikov arc~$\gamma'$ in
  $D_{n+1}$ and a component~$J$ of $L'\setminus\gamma'$ which is an
  arc whose union with the segment of $\gamma'$ bounded by its
  endpoints forms a simple closed curve~$C$ bounding an unpunctured
  disk~$D$. Since~$L$ is a minimal representative of~$\cL$ and
  $\gamma'$ is a Dynnikov arc in $D_{n+2}$, the disk~$D$ must
  contain the filled in puncture~$i+1$. Again, since $L$ is minimal,
  if $\gamma'$ is to the right (respectively left) of puncture~$i+1$,
  then~$J$ does not intersect the Dynnikov arc $\beta_i$ (respectively
  $\beta_{i+1}$); hence $L$ has a left (respectively right) loop
  component in $\Delta_i$, so that $b_i<0$ (respectively
  $b_i>0$). This is the required contradiction.

\item We have $\beta_0' = \beta_0 =0$ and $\beta_{n-1}' =
  \beta_n=0$ (or, if $i=n-1$, then $\beta_{n-1}' = \beta_{n-1} = 0$,
  since $b_{n-1}=0$ and $\beta_n=0$), so $\cL'$ is central.
\end{enumerate}

\end{proof}

\subsubsection{Erasing elementary components}
This move is applied when a minimal representative of the lamination
contains components which are elementary curves. We erase these
elementary components, thereby simplifying the lamination.

\begin{lem}
\label{lem:erase-simple}
Let $(a;\,b)\in\Z^{2n}$ be the Dynnikov coordinates of a central
integral lamination~$\cL$ with $n>3$.

Suppose that $b_j\not=0$ for all~$j$, so that $i=i(a;\,b)>0$ and $M =
\min(-b_{i-1},\,b_i)>0$; and that $a_{i-1}=a_i$. Let
$(a';\,b')\in\Z^{2n}$ be defined by $a'=a$; $b'_{i-1} = b_{i-1}+M$;
$b'_i = b_i-M$; and $b'_j = b_j$ for all~$j\not=i-1,\,i$.

Then $(a';\,b')$ are the Dynnikov coordinates of a central integral
lamination with $X(a';\,b') = X(a;\,b) - M$ and $\cC(a';\,b') <
\cC(a;\,b)$. 
\end{lem}
\begin{proof}
Let $L$ be a minimal representative of~$\cL$. By
Lemma~\ref{lem:elementary}, $L$ has $M$ components which are
elementary curves about punctures $i$ and $i+1$. Moreover, $L$ has
other components besides these, since all of the $b_j$ are non-zero
(an integral lamination consisting entirely of elementary curves about
these punctures would have $b_j=0$ for all $j\not=i-1,\,i$). Therefore
the union of simple closed curves~$L'$ obtained from~$L$ by erasing
these elementary curves is a minimal representative of a central
integral lamination~$\cL'$.

Erasing the elementary curves reduces the number of intersections with
the arcs $\alpha_{2i-3}$, $\alpha_{2i-2}$, $\alpha_{2i-1}$, and
$\alpha_{2i}$ by~$M$; and the number of intersections with the arc
$\beta_i$ by $2M$. Therefore~$\cL'$ has Dynnikov coordinates
$(a';\,b')$ as given in the statement of the lemma.

Clearly $X(a';\,b') = X(a;\,b)-M$; and $n(a';\,b')=n(a;\,b)$, while
$\sum_{i=0}^{n-1}|b_i'| = \sum_{i=0}^{n-1}|b_i|-2M$, so that
$\cC(a',b') < \cC(a,b)$.
\end{proof}

\subsubsection{Untwisting}
This move is applied when there are two consecutive punctures, with a
left loop about the left puncture and a right loop about the right
puncture, but no elementary curves about these two
punctures. Applying an appropriate braid generator simplifies the
lamination.

\begin{lem}
\label{lem:untwist}
Let $(a;\,b)\in\Z^{2n}$ be the Dynnikov coordinates of a central
integral lamination~$\cL$ with $n>3$. 

Suppose that $b_j\not=0$ for all~$j$, so that $i=i(a;\,b)>0$; and that
$a_{i-1}\not=a_i$. Let $(a';\,b')\in\Z^{2n}$ be defined by $a_j' =
a_j$ and $b_j'=b_j$ for all $j\not=i-1,\,i$, and:
\begin{description}
\item[Case Ia): $0<a_{i-1}-a_i$ \quad and \quad $b_i-b_{i-1} \le a_{i-1}-a_i$.] 
\[
\begin{aligned}
a_{i-1}' &= a_i - b_{i-1},\\
a_i' &= a_{i-1}-b_i,\\
b_{i-1}' &= b_i,\\
b_i' &= b_{i-1}.
\end{aligned}
\]

\item[Case Ib): $0<a_{i-1}-a_i<b_i-b_{i-1}$.] 
\[
\begin{aligned}
a_{i-1}' &= \min(a_i-b_{i-1},\,\, a_{i-1}),\\
a_i' &= \max(a_{i-1}-b_i,\,\, a_i),\\
b_{i-1}' &= b_{i-1} + (a_{i-1}-a_i),\\
b_i' &= b_i - (a_{i-1}-a_i).
\end{aligned}
\]

\item[Case IIa): $0<a_i-a_{i-1}$ \quad and \quad $b_i-b_{i-1} \le a_i-a_{i-1}$.] 
\[
\begin{aligned}
a_{i-1}' &= a_i+b_{i-1},\\
a_i' &= a_{i-1}+b_i,\\
b_{i-1}' &= b_i,\\
b_i' &= b_{i-1}.
\end{aligned}
\]

\item[Case IIb): $0<a_i-a_{i-1} < b_i-b_{i-1}$.] 
\[
\begin{aligned}
a_{i-1}' &= \max(a_{i-1},\,\, a_i+b_{i-1}),\\
a_i' &= \min(a_i,\,\, a_{i-1}+b_i),\\
b_{i-1}' &= b_{i-1} + (a_i-a_{i-1}),\\
b_i' &= b_i - (a_i-a_{i-1}).
\end{aligned}
\]
\end{description}

\medskip

Then $(a';\,b')$ are the Dynnikov coordinates of a central integral
lamination with $X(a';\,b')=X(a;\,b)$ and $\cC(a';\,b')<\cC(a;\,b)$.
\end{lem}

\begin{proof}
The coordinates $(a';\,b')$ in the statement of the lemma are obtained
from~$(a;\,b)$ by applying the braid generator $\sigma_i^{-1}$
(case~I) or $\sigma_i$ (case~II), using the inequalities pertaining to
each case to resolve some of the maxima and minima
in~(\ref{eq:pos-update}) and~(\ref{eq:neg-update}). Therefore
$(a';\,b')$ are the Dynnikov coordinates of a central integral
lamination with $X(a';\,b') = X(a;\,b)$, and it only remains to show
that the complexity has decreased.

In all cases we have $n(a';\,b') = n(a;\,b)$. In cases Ia) and IIa) we
also have $\sum_{i=0}^{n-1}|b_i'| = \sum_{i=0}^{n-1}|b_i|$: however
$i(a';\,b') = i(a;\,b)-1$, since the first positive component of $b'$
is $b'_{i-1}$. Therefore $\cC(a';\,b')<\cC(a;\,b)$.

In case Ib) we have that $b_{i-1}<0$, $b_i>0$, and $0<a_{i-1}-a_i<
b_i-b_{i-1}$, and we proceed by considering cases.
\begin{enumerate}[a)]
\item If $a_{i-1}-a_i \le \min(-b_{i-1},\, b_i)$ then 
\[
\begin{aligned}
|b_{i-1}'| + |b_i'| &= (-b_{i-1} - (a_{i-1}-a_i)) + (b_i -
(a_{i-1}-a_i)) \\
&= |b_{i-1}| +  |b_i| - 2(a_{i-1}-a_i) < |b_{i-1}| + |b_i|.
\end{aligned}
\]
  
\item If $-b_{i-1} < a_{i-1}-a_i \le b_i$ then
\[
\begin{aligned}
|b_{i-1}'| + |b_i'| &= (b_{i-1} + (a_{i-1}-a_i)) + (b_i -
(a_{i-1}-a_i))\\
&= b_{i-1} + b_i < |b_{i-1}| + |b_i|.
\end{aligned}
\]

\item If $b_i < a_{i-1}-a_i \le -b_{i-1}$ then
\[
\begin{aligned}
|b_{i-1}'| + |b_i'| &= (-b_{i-1} - (a_{i-1}-a_i)) + (-b_i +
(a_{i-1}-a_i))\\
&= -b_{i-1} - b_i < |b_{i-1}| + |b_i|.
\end{aligned}
\]

\item If $\max(-b_{i-1},\, b_i) < a_{i-1}-a_i$ then
\[
\begin{aligned}
|b_{i-1}'| + |b_i'| &= (b_{i-1} + (a_{i-1}-a_i)) + (-b_i +
(a_{i-1}-a_i)) \\
&< b_{i-1} - b_i + 2(b_i-b_{i-1}) = b_i-b_{i-1} = |b_{i-1}| + |b_i|.
\end{aligned}
\]
\end{enumerate}
Therefore $\cC(a';\,b') < \cC(a;\,b)$. A similar argument applies in case~IIb).
\end{proof}

\subsection{Statement of the algorithm}
\label{sec:alg-state}
Algorithm~\ref{alg:algorithm} below computes the number of 
components of an integral lamination $\cL\in\cL_n$. We assume that
$n>3$, since otherwise the number of components is given by
Lemma~\ref{lem:3-punctures}. The algorithm works with a pair
$((a;\,b),\, Y)$, where $(a;\,b)$ are extended Dynnikov coordinates
and~$Y$ is a non-negative integer which counts the number of
elementary curve components which have been erased: the quantity
$X(a;\,b) + Y$ remains constant throughout.

\begin{alg}
\label{alg:algorithm}
Let $(a;\,b)\in\Z^{2n-4}$ be the Dynnikov coordinates of an integral
lamination $\cL\in\cL_n$, with $n>3$.
\begin{description}
\item[Step 1] Replace $(a;\,b)$ with $(a;\,b)\in\Z^{2n}$ given
  by~(\ref{eq:extend}). Set $Y=0$ and input the pair $((a;\,b),\,Y)$
  to Step~2.

\item[Step 2] \textbf{If} $b_i=0$ for some~$i$, then let $(a';\,b')$
  be given by {\em Filling in a puncture}
  (Lemma~\ref{lem:fill-puncture}). If \mbox{$n(a';\,b')=3$} then input
  $((a';\,b'),\,Y)$ to Step~5: if $n(a';\,b')>3$ then input
  $((a';\,b'),\,Y)$ to Step~2.

\textbf{Otherwise}, input $((a;\,b),\,Y)$ to Step~3.

\item[Step 3] Let $i=i(a;\,b)$. \textbf{If} $a_{i-1}=a_i$, then let
  $(a';\,b')$ be given by {\em Erasing elementary components}
  (Lemma~\ref{lem:erase-simple}), and input $((a';\,b'),\,
  Y+\min(-b_{i-1},\,b_i))$ to Step~2.

\textbf{Otherwise}, input $((a;\,b),\,Y)$ to Step~4.

\item[Step 4] Let $(a';\,b')$ be given by {\em Untwisting}
  (Lemma~\ref{lem:untwist}). Input $((a';\,b'),\,Y)$ to Step~3 in cases
  Ia) and IIa) of the Lemma, or to Step~2 in cases Ib) and IIb).

\item[Step 5] Since $n(a;\,b)=3$, we have $(a;\,b) =
  (0,a_1,0;\,b_0,b_1,b_2)$. The number of components of the original
  integral lamination is given by
\[
\gcd(a_1,b_1) + Y + \min(-b_0,\, b_2,\, -|a_1|-b_0-b_1^+).
\]
\end{description}
\end{alg}

\begin{proof}
It is immediate from Lemmas~\ref{lem:fill-puncture},
\ref{lem:erase-simple}, and \ref{lem:untwist} that the quantity $X(a\,;b) +
Y$ remains constant throughout the algorithm, and that $\cC(a\,;b)$
decreases each time one of the moves is applied. The algorithm
therefore terminates (i.e.\ reaches Step~5), since there are no
infinite strictly decreasing sequences in $\N^3$; and the
number of components of the starting integral lamination is equal to
$X(a;\,b) + Y$, where $(a;\,b)=(0,a_1,0;\, b_0,b_1,b_2)$ and $Y$ are
the inputs to Step~5. It therefore only remains to show that
\[
X(0,a_1,0;\, b_0,b_1,b_2) = \gcd(a_1,b_1) + \min(-b_0,\, b_2,\, -|a_1|
- b_0 - b_1^+).
\]
Now $(0,a_1,0;\,b_0,b_1,b_2)$ are the Dynnikov coordinates of a central
lamination on $D_5$. Let~$L$ be a minimal representative. We shall
show that the number~$Z$ of components of~$L$ which bound a disk
containing the three central punctures is given by
$Z=\min(-b_0,\,b_2,\,-|a_1|-b_0-b_1^+)$. This will complete the proof, since
erasing these components and filling in the two end punctures yields
a representative of an integral lamination on~$D_3$ with Dynnikov
coordinates $(a_1;\,b_1)$, which has $\gcd(a_1,\,b_1)$
components by Lemma~\ref{lem:3-punctures}.

Now $Z$ is the minimum of the number of left loop components
in~$\Delta_0$; the number of right loop components in~$\Delta_2$; the
number of above components in~$\Delta_1$; and the number of below
components in~$\Delta_1$. The first two of these numbers are $-b_0$
and $b_2$; and the third and fourth are $\alpha_1 - |b_1|$ and $\alpha_2
- |b_1|$. 

Suppose first that $b_1\ge 0$. Since the lamination is central, we
have $\beta_0=0$, and hence $\beta_1 = -2b_0$
by~(\ref{eq:dynn-coords}). Then~(\ref{eq:alpha}) gives that $\alpha_1
= -a_1-b_0$ and $\alpha_2 = a_1-b_0$. Therefore
\[
Z = \min(-b_0,\, b_2,\, -a_1-b_0-b_1,\, a_1-b_0-b_1) = \min(-b_0,\, b_2,\,
-|a_1| - b_0 - b_1^+)
\]
as required. A similar argument applies in the case $b_1\le 0$, when
$\alpha_1 = -a_1-b_0-b_1$ and $\alpha_2 = a_1-b_0-b_1$.
\end{proof}

\subsection{Complexity of the algorithm}
In this section we analyse the complexity of the algorithm, when
applied to an integral lamination $\cL\in\cL_n$ with Dynnikov
coordinates $(a;\,b)\in\Z^{2n-4}$. 

Write $M=\sum_{i=1}^{n-2}(|a_i|+|b_i|)$. By an {\em arithmetic
  operation} we mean adding, subtracting, comparing, taking the maximum, or
taking the minimum of two integers. As we will see, these integers
have absolute value $O(n^2M^2)$ thoughout the algorithm,
so that the cost of each arithmetic operation is logarithmic in $n$
and $M$.

Steps~1 and~5 are each carried out only once in the algorithm. Step~1
involves $O(n^2)$ arithmetic operations, while Step~5 involves
$O(\log(n^2M^2))$ arithmetic operations (to calculate the greatest
common divisor). Observe that the Dynnikov coordinates $(a';\,b')$
produced by Step~1 satisfy
\[
M' := \sum_{i=0}^{n-1}(|a'_i| + |b'_i|) \le M + 2
\max_{1\le k\le n-2}
\left(
|a_k| + b_k^+ + \sum_{j=1}^{k-1} b_j
\right) + \sum_{j=1}^{n-2}|b_j| = O(M).
\]

Now consider the main body of the algorithm, consisting of Steps~2,~3,
and~4. This can naturally be regarded as a loop: at each iteration,
$O(n)$ arithmetic operations are carried out to scan the~$b$
coordinates and identify whether some $b_i=0$; and, if not, to find
$i=i(a;\,b)$ and to determine whether or not $a_{i-1}=a_i$. According
to the results of these tests, one of the three moves {\em Filling in
  a puncture}, {\em Erasing elementary components}, or {\em
  Untwisting} is carried out.

Each of these three moves involves $O(1)$ arithmetic operations; and
none of them increases $\sum_{i=1}^{n-1}|b_i|$.

 {\em Filling in a puncture} is carried out $O(n)$ times during the
 course of the algorithm.

{\em Erasing elementary components} strictly decreases
$\sum_{i=1}^{n-1}|b_i|$, and so is carried out $O(M)$ times during the
course of the algorithm.

{\em Untwisting} strictly decreases $\sum_{i=1}^{n-1}|b_i|$ in cases
Ib) and IIb), so these cases are carried out $O(M)$ times during the
course of the algorithm. Cases Ia) and IIa) leave
$\sum_{i=1}^{n-1}|b_i|$ constant. However, since they decrease
$i(a;\,b)$ by exactly~1, these cases are repeated $O(n)$ times
before either {\em Erasing elementary components} or Case Ib) or IIb)
is applied. Moreover, no scanning of the~$b$ coordinates is necessary
between successive applications of these cases.

The main body of the algorithm therefore involves $O(n^2M)$ arithmetic
operations. Observing that none of the moves increases
$\max_{i=0}^{n-1}|a_i|$ by more than $\sum_{i=1}^{n-1}|b_i|$ (which
remains $\le M$ throughout the algorithm), we see that the maximum
size of the integers involved in arithmetic operations is $O(n^2M^2)$.  

We therefore have the following result:
\begin{lem}
\label{lem:complexity}
Let $(a;\,b)$ be the Dynnikov coordinates of an integral lamination
$\cL\in\cL_n$ with $n>3$, and write $M =
\sum_{i=1}^{n-2}(|a_i|+|b_i|)$.  Then applying
Algorithm~\ref{alg:algorithm} to~$\cL$ requires $O(n^2M)$ arithmetic
operations, each carried out on a pair of integers of sizes
$O(n^2M^2)$.
\end{lem}

\subsection{An example}
\label{sec:example}
In this section we use Algorithm~\ref{alg:algorithm} to compute the
number of components of the integral lamination $\cL\in\cL_6$ with
Dynnikov coordinates $\rho(\cL) =
(-1,\,-2,\,-2,\,1\,;\,\,-1,\,2,\,-2,\,2)$. The successive moves are
illustrated in Figure~\ref{fig:example}. 
\begin{enumerate}[1.]
\item {\em Extend coordinates}: apply~(\ref{eq:extend}) to
  replace the coordinates with \[(a;\,b) =
  (0,\,-1,\,-2,\,-2,\,1,\,0\,;\,\,-3,\,-1,\,2,\,-2,\,2,\,2),\] and
  input $((a;\,b),\,0)$ to the main algorithm (Step~2).

\item {\em Untwisting}: we have $b_i\not=0$ for all~$i$, so we proceed
  to Step~3. The first positive $b$ coordinate is $b_2$, so we have
  $i(a;\,b) = 2$. Since $a_1\not=a_2$ we proceed to Step~4. 

Since $a_1-a_2=1>0$ and $b_2-b_1=3>a_1-a_2$ we are in Case Ib). We get
\[(a;\,b) = (0,\,-1,\,-2,\,-2,\,1,\,0\,;\,\,-3,\,0,\,1,\,-2,\,2,\,2),\]
and input $((a;\,b),\,0)$ to Step~2.

\item {\em Fill in puncture 2}: since $b_1=0$ we fill in
  puncture~2. We get 
\[(a;\,b) = (0,\,-2,\,-2,\,1,\,0\,;\,\,-3,\,1,\,-2,\,2,\,2),\]
 and input
  $((a;\,b),\,0)$ to Step~2.

\item {\em Untwisting}: we have $b_i\not=0$ for all~$i$, so we proceed
  to Step~3. The first positive $b$ coordinate is $b_1$, so we have
  $i(a;\,b) = 1$. Since $a_0\not=a_1$ we proceed to Step 4.

Since $a_0-a_1 = 2 > 0$ and $b_1-b_0=4>a_0-a_1$ we are again in Case
Ib). We get
\[
(a;\,b) = (0,\,-1,\,-2,\,1,\,0\,;\,\,-1,\,-1,\,-2,\,2,\,2),
\]
and input $((a;\,b),\,0)$ to Step~2.

\item {\em Untwisting}: we have $b_i\not=0$ for all~$i$, so we proceed
  to Step~3. The first positive~$b$ coordinate is $b_3$, so we have
  $i(a;\,b)=3$. Since $a_2\not=a_3$, we proceed to Step~4.

Since $a_3-a_2 = 3>0$ and $b_3-b_2 = 4>a_3-a_2$ we are in Case
IIb). We get
\[
(a;\,b) = (0,\,-1,\,-1,\,0,\,0\,;\,\,-1,\,-1,\,1,\,-1,\,2).
\]

\item {\em Erasing an elementary component}: we have $b_i\not=0$ for
  all~$i$, so we proceed to Step~3. The first positive $b$ coordinate
  is $b_2$, so we have $i(a;\,b)=2$. Since $a_1=a_2$, we can erase
  elementary components about punctures $2$ and $3$. Since $\min(-b_1,
  b_2)=1$, there is~1 such component. We get
\[
(a;\,b) = (0,\,-1,\,-1,\,0,\,0\,;\,\,-1,\,0,\,0,\,-1,\,2),
\]
and input $((a;\,b),\,1)$ to Step~2.

\item {\em Fill in punctures 2 and 3}: since $b_1=b_2=0$ we fill in
  punctures~2 and~3 (following the algorithm strictly, we first fill in
  puncture~2, and then fill in the new puncture~2, which is the
  previous puncture~3). We get 
\[
(a;\,b) = (0,\,0,\,0\,;\,\,-1,\,-1,\,2).
\]
Since $n(a;\,b)=3$, we input $((a;\,b),\,1)$ to Step~5.

\begin{figure}
\begin{center}
\includegraphics[height=0.75\textheight]{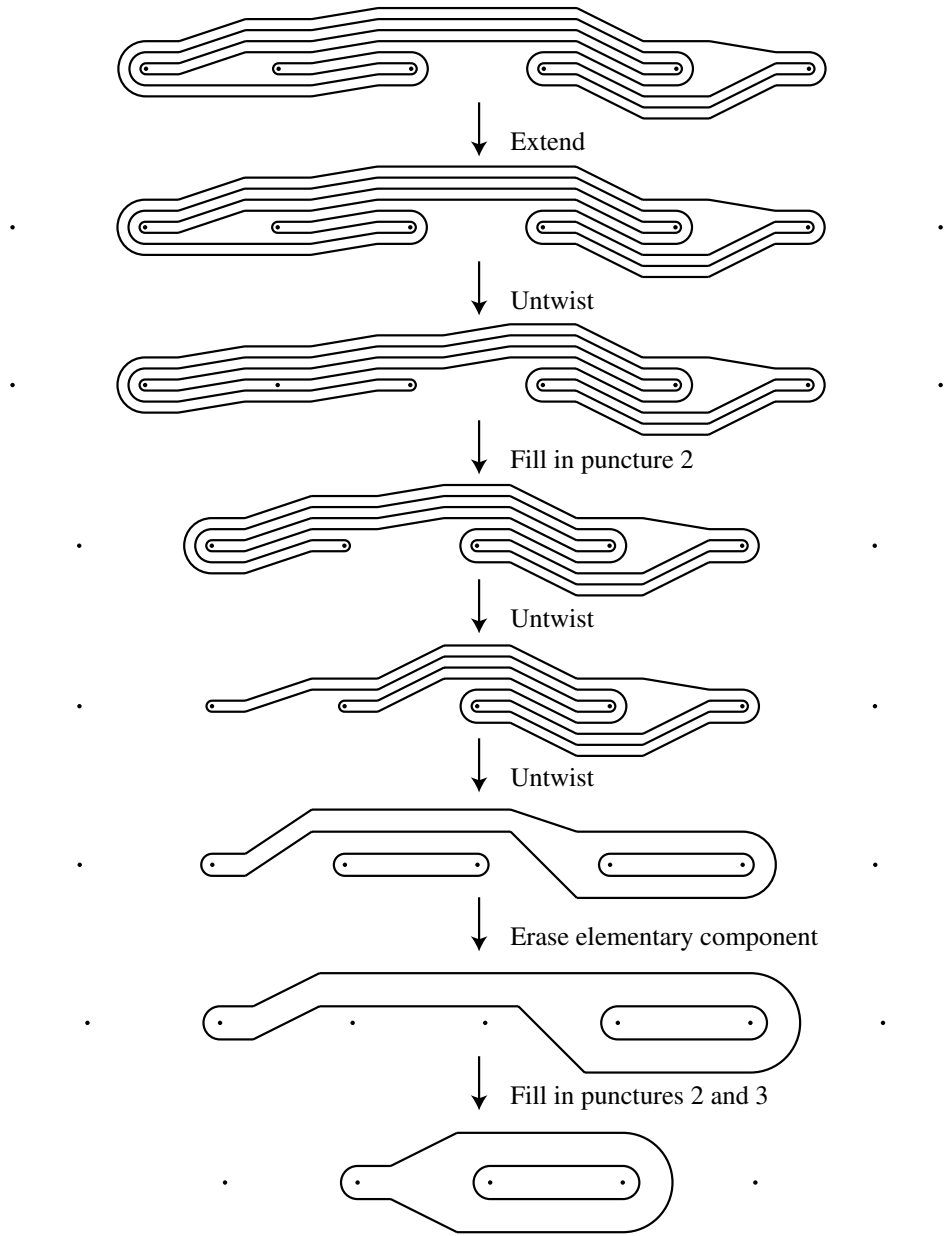}
\caption{An example of the algorithm}
\label{fig:example}
\end{center}
\end{figure}

\item {\em Number of components}: we have $Y=1$; $\gcd(a_1,b_1) =
  \gcd(0,-1) = 1$; and $\min(-b_0,\,b_2,\,-|a_1|-b_0-b_1^+) =
  \min(1,\,2,\,1) = 1$ (corresponding to the fact that our final
  lamination has one component bounding a disk containing the three
  central punctures).

Therefore $X(-1,\,-2,\,-1,\,0\,;\,\,-1,\,2,\,1,\,-1) = 3$.
\end{enumerate}

\bibliographystyle{plainnat}

\begin{thebibliography}{99}

\bibitem{emil2}
E.~Artin, \emph{Theory of braids}, Ann. of Math. (2) \textbf{48} (1947),
  101--126.

\bibitem{birman}
J.~S. Birman, \emph{Braids, links, and mapping class groups}, Princeton
  University Press, Princeton, N.J., 1974, Annals of Mathematics Studies, No.
  82.

\bibitem{or08}
P.~Dehornoy, I.~Dynnikov, D.~Rolfsen, and B.~Wiest, \emph{Ordering braids},
  Mathematical Surveys and Monographs, vol. 148, American Mathematical Society,
  Providence, RI, 2008.

\bibitem{D02}
I.~Dynnikov, \emph{On a {Y}ang-{B}axter mapping and the {D}ehornoy ordering},
  Uspekhi Mat. Nauk \textbf{57} (2002), no.~3(345), 151--152.

\bibitem{dynn-wiest}
I.~Dynnikov and B.~Wiest, \emph{On the complexity of braids}, J. Eur. Math.
  Soc. (JEMS) \textbf{9} (2007), no.~4, 801--840.

\bibitem{Haas-Susskind}
A.~Haas and P.~Susskind, \emph{The connectivity of multicurves determined by
  integral weight train tracks}, Trans. Amer. Math. Soc. \textbf{329} (1992),
  no.~2, 637--652. 

\bibitem{paper1}
T.~Hall and S.~{\"O}. Yurtta{\c{s}}, \emph{On the topological entropy of
  families of braids}, Topology Appl. \textbf{156} (2009), no.~8, 1554--1564.

\bibitem{M06}
J-O. Moussafir, \emph{On computing the entropy of braids}, Funct. Anal. Other
  Math. \textbf{1} (2006), no.~1, 37--46.

\bibitem{penner}
R.~C. Penner and J.~L. Harer, \emph{Combinatorics of train tracks}, Annals of
  Mathematics Studies, vol. 125, Princeton University Press, Princeton, NJ,
  1992.

\bibitem{paper2}
S.~{\"O}. Yurtta{\c{s}}, \emph{Geometric intersection of curves on punctured
  disks}, Journal of the Mathematical Society of Japan \textbf{65} (2013),
  no.~4, 1554--1564.


\end{thebibliography}

\end{document}